\documentclass{article}

\usepackage[english]{babel}

\usepackage[letterpaper,top=2cm,bottom=2cm,left=3cm,right=3cm,marginparwidth=1.75cm]{geometry}

\usepackage{amsmath}
\usepackage{amssymb}
\usepackage{amsthm}
\usepackage{graphicx}
\usepackage[colorlinks=true, allcolors=blue]{hyperref}

\newtheorem{theorem}{Theorem}
\newtheorem{proposition}[theorem]{Proposition}
\newtheorem{lemma}[theorem]{Lemma}
\newtheorem{corollary}[theorem]{Corollary}
\newtheorem{question}[theorem]{Question}
\theoremstyle{definition}
\newtheorem{definition}[theorem]{Definition}
\theoremstyle{remark}
\newtheorem{observation}[theorem]{Observation}
\newtheorem{remark}[theorem]{Remark}

\newcommand{\tuple}[1]{\left\langle#1\right\rangle}
\newcommand{\concat}{{^\frown}}
\newcommand{\term}[1]{\textbf{#1}}

\title{Elementary Submodels, Coding Strategies, and an Infinite Real Number Game}
\author{Will Brian, Steven Clontz}

\begin{document}
\maketitle

\begin{abstract}
Matthew Baker investigated in \cite{10.2307/27643064}
an elegant infinite-length game that may be used to
study subsets of real numbers. We present two accessible
examples of how an important technique
from set theory or a different technique from 
infinite game theory may be used to
answer Baker's question on whether this game provides
a precise characterization for countable subsets of real
numbers, and connect this game to the well-studied
Banach-Mazur game from topology.
\end{abstract}

\section{Introduction}

Let \(\omega=\{0,1,2,\dots\}\) denote the first infinite ordinal.
The following game was named by Matthew Baker in \cite{10.2307/27643064},
based upon a game appearing in Mathematics Magazine problem \#1542 
\cite{doi:10.1080/0025570X.1998.11996621}.

\begin{definition}
Let \(W\) be a set of real numbers, called the \term{payoff set}.
During each round \(n<\omega\) of the \term{Cantor Game} \(CG(W)\),
Alice chooses a legal real number \(a_n\), followed by Bob choosing
a legal real number \(b_n\), where a number is legal provided it is
strictly greater than all previous choices of Alice and strictly less than
all previous choices of Bob.

After \(\omega\)-many rounds, Alice is said to have won the game provided
\(\displaystyle\lim_{n\to\infty}a_n\in W\); Bob wins otherwise.
\end{definition}

\begin{definition}
A strategy for a player in a game is said to be \term{winning} provided
any counter-strategy for the opponent is defeated by it. That is,
a strategy is not winning provided there exists a \term{successful}
counter-strategy for the opponent that defeats it.
\end{definition}

Note in particular that counter-strategies generally are defined
in terms of the strategy they counter. (If not, then the counter-strategy
is actually just a strategy!)

The Cantor Game is of particular interest as a very accessible introduction
to the area of infinite-length games. Such games are of importance in
both set theory and topology; see \cite{10.2307/44237047} for a classic
survey.
However, in contrast with the standard examples,
the Cantor Game requires no mathematical experience
beyond a freshman calculus
course to appreciate. To illustrate why such games are convenient tools
to characterize mathematical properties, in \cite{10.2307/27643064} Baker observed
the following.

\begin{definition}
A subset of a topological space is \term{perfect} if it is non-empty, closed,
and equals the set of its limit points.
\end{definition}

\begin{proposition}\label{bakerMainAlice}
If \(W\subseteq\mathbb R\) contains a perfect set, then Alice has a
winning strategy.
\end{proposition}

\begin{proposition}\label{bakerMainBob}
If \(W\subseteq\mathbb R\) is countable, then Bob has a winning strategy
for \(CG(W)\).
\end{proposition}

In particular, \(\mathbb R\) (which is perfect) must be uncountable!
Baker left open the converses of these implications; in 2017 his student
LaDue demonstrated the following.

\begin{theorem}[\cite{https://doi.org/10.48550/arxiv.1701.09087}]
Alice has a winning strategy in \(CG(W)\) if and only if
\(W\) contains a perfect set.
\end{theorem}

LaDue then observes that under the Axiom of Determinacy, which implies that every uncountable subset of $\mathbb R$ contains a perfect set,
Bob has a winning strategy in \(CG(W)\) if and only if
\(W\) is countable.
This paper shows this axiom is unnecessary: even without assuming the Axiom of Determinacy, Bob has a winning strategy in \(CG(W)\) if and only if
\(W\) is countable.
In doing so, we
demonstrate two mathematical concepts not typically accessible
to a general mathematical audience: elementary submodels and
limited-information strategies.

But before we do, let's we define a game equivalent to \(CG(W)\) that's
more convenient to study.

\begin{definition}
Let \(W\) be a payoff set of real numbers.
Then the \term{Baker Game} \(BG(W)\) proceeds identically to
the Cantor Game, except Alice wins whenever 
there remains a legal \(w\in W\) after \(\omega\)-many rounds.
\end{definition}

In other words, \(CG(W)\) and \(BG(W)\) are the same except when it comes to determining the winner at the end of the game. In \(CG(W)\), Alice wins if and only if $\displaystyle \lim_{n \rightarrow \infty}a_n$ is a point of $W$, whereas in \(BG(W)\) Alice wins if and only if there is a point of $W$ betweeen $\displaystyle \lim_{n \rightarrow \infty}a_n$ and $\displaystyle \lim_{n \rightarrow \infty}b_n$ (inclusive).

We say that a strategy (resp. counter-strategy)
modification is an \term{improvement}
if the modification preserves winning strategies
(resp. successful counter-strategies).
A rigorous verification of the following results is beyond
the scope of this paper, but they all follow essentially
from the observations that (1) \(\displaystyle\lim_{n\to\infty}a_n\) is
legal at the end of the game, (2) by choosing smaller legal numbers,
Bob reduces Alice's possible moves without providing her
additional ways to win, and thus (3) by playing numbers within \(\frac{1}{n}\) of Alice's last move each
round, Bob can ensure without loss that 
\(\displaystyle \lim_{n\to\infty}a_n = \lim_{n\to\infty}b_n\) 
at the end of the game.

\begin{lemma}\label{bobSmaller}
Any (counter-)strategy for Bob in the Cantor Game (resp. Baker Game) may be
improved by replacing each choice by a smaller legal number.
\end{lemma}


\begin{theorem}
Each player has a winning strategy in the Baker Game if and only if
they have a winning strategy in the Cantor Game.
\end{theorem}



\begin{observation}\label{bobStrat}
By Lemma \ref{bobSmaller} we may assume without loss of
generality that Bob only plays rational numbers in either game.
\end{observation}

We will also make use of this fact about the real numbers.

\begin{definition}
A subset of the reals if \term{well-founded} if it contains
no infinite decreasing sequence.
\end{definition}

\begin{theorem}
Every well-founded subset of the reals is countable.
\end{theorem}

\section{Elementary Submodels}

To illustrate the use of elementary submodels, it's helpful first
to demonstrate a na\"{i}ve (and ultimately incorrect) proof of our main theorem.

\begin{remark}\label{willPartial}
Let \(\sigma\) be an arbitrary strategy for Bob,
where \(\sigma(t)\) is Bob's next choice given the sequence
of prior choices \(t=\tuple{a_0,\dots,a_{n-1}}\) by Alice.
By Observation \ref{bobStrat}, we may (and do) assume \(\sigma(t)\) is a rational number for any sequence $t$ of plays by Alice.

For each sequence of choices \(t=\tuple{a_0,\dots,a_{n-1}}\) by Alice, let us say $x$ is legal after $t$ if it is legal for Alice to play $x$ next, i.e., if $a_{n-1} < x < \sigma(t)$.
Define
\[E_t=\left\{x\text{ legal after }t:\forall a
\big(a\in(a_{n-1},x)\Rightarrow x\geq \sigma(t\concat\tuple{a})\big)\right\}.\]
In other words, $E_t$ contains those $x$ that are legal after $t$, and  certain to be eliminated in round $n$ of the game, because no matter what number $a$ Alice chooses to play, $x$ will either be $\leq$ Alice's play $a$ or $\geq$ Bob's play $\sigma(t\concat\tuple{a})$.

We claim that this set contains no infinite decreasing sequence
bounded above the last move \(a_{n-1}\) of Alice.
To see
this, let \(S\) be such a sequence and let \(a=\inf S\); then
\(a<s<\sigma(t\concat\tuple a)\)
for some \(s\in S\), showing \(s\not\in E_{t}\).
It follows that \(E_{t}\cap[a_{n-1}+\varepsilon,\infty)\) is well-founded and
therefore countable for each \(\varepsilon>0\),
and thus \(E_{t}\) is countable.

If $x \notin E_t$ for every $t$, then Alice can play to ensure that $x$ is between $\displaystyle \lim_{n \rightarrow \infty}a_n$ and $\displaystyle \lim_{n \rightarrow \infty}b_n$ at the end of the game.
Indeed, given a sequence \(t=\tuple{a_0,\dots,a_{n-1}}\) of Alice's moves
keeping \(x\) legal against \(\sigma\), Alice's next move
can be any \(a_n\in(a_{n-1},x)\) with \(x<\sigma(t\concat\tuple{a_n})\) (which exists because $x \notin E_t$).
Then \(x\) remains legal as \(a_n<x<\sigma(t\concat\tuple{a_n})\).
\end{remark}

So far, so good.
The na\"{i}ve idea to complete this proof is to require that Alice only plays rational numbers at any point in the game.
There are only countably-many finite sequences of rationals, so this means there are only countably many sets $E_t$ to consider. As
\(E_t\) is countable for every $t$, this means that for any uncountable payoff set
\(W\), there is some \(x\in W\) such that $x \notin E_t$ for every sequence
of rationals \(t\). As we just demonstrated, this implies that Alice can play the game to ensure that $x$ is between $\displaystyle \lim_{n \rightarrow \infty}a_n$ and $\displaystyle \lim_{n \rightarrow \infty}b_n$ at the end of the game.

The requirement that Alice only play rational numbers may seem innocuous: after all, any strategy of Alice can be improved by playing a slightly smaller rational each round, so if Alice has a winning strategy, then she has one using only rational numbers (just like Bob).

However, we are not dealing with a strategy of Alice's, but a strategy of Bob's. It is nonsense, in this context, to ask that Alice play with an improved strategy using only rationals, because we are considering whether Alice has a counter-strategy that defeats \(\sigma\).
The real issue here is the part of Remark \ref{willPartial} proving $E_t$ is countable. This proof
requires Alice to choose a particular number \(a_n\in(a_{n-1},x)\)
where \(x<\sigma(t\concat\tuple{a_n})\) witnessing that \(x\not\in E_t\):
and there's no guarantee that this ``witness'' is a rational number! 
Thus it seems that our proof that $E_t$ is countable is incompatible with requiring Alice to play only rational numbers. 

Of course, the rational numbers are not the only countable set in town. Perhaps, to get around this obstacle, we can expand the set of Alice's possible moves to include all the ``witnesses'' mentioned in the previous paragraph. If the expanded set of possible plays remains countable, we will still have only countably many sets $E_t$, and our na\"{i}ve proof idea can be made to work.


Finding a large enough -- but still countable -- set of possible plays for Alice is precisely where elementary submodels come into the picture. Their part in the proof begins with Theorem~\ref{LSthm} below, a consequence of
the downward L\"owenheim–Skolem theorem.

Given a large, uncountable set $H$, we will consider the structure $(H,\in)$ as a possible model for some or all of the $\mathsf{ZFC}$ axioms of set theory. A set $M \subseteq H$ is an \textbf{elementary submodel} of $H$ if for any $a_1,a_2,\dots,a_n \in M$, and any $n$-place formula $\varphi$ of first order logic,
$$(M,\in) \Vdash \varphi(a_1,a_2,\dots,a_n) \ \ \ \Leftrightarrow \ \ \ (H,\in) \Vdash \varphi(a_1,a_2,\dots,a_n).$$
In other words, $M$ is an elementary submodel of $H$ if the structures $(M,\in)$ and $(H,\in)$ agree on the truth or falsity of every statement that can be expressed in first order logic and mentions only members of $M$.
We may think of $(H,\in)$ as a structure modeling some important statements (such as $\mathsf{ZFC}$), and $(M,\in)$ as a microcosm for this structure, that models all the same statements, and is also countable.

\begin{theorem}\label{LSthm}
Given an uncountable set $H$ and a countable set \(C \subseteq H\),
there is a countable set $M$ such that \(C \subseteq M \subseteq H\), and $M$ is an elementary submodel of $H$.
\end{theorem}

One way to view this theorem is to think in terms of ``witnesses'' as in the discussion at the beginning of this section. Given some existential statement $\varphi$ of the form $\exists z \,\psi(z,a_1,\dots,a_n)$ concerning some members $a_1,a_2,\dots,a_n$ of $M$, if $H$ contains a witness to this statement, then so must $M$. More precisely, if  $(H,\in) \Vdash \exists z\,\psi(z,a_1,\dots,a_n)$ (which is equivalent to there being some $w \in H$ with $(H,\in) \Vdash \psi(w,a_1,\dots,a_n)$), then $(M,\in) \Vdash \exists z\,\psi(z,a_1,\dots,a_n)$ (equivalently, $M$ must contain some $w'$ such that $(M,\in) \Vdash \psi(w',a_1,\dots,a_n)$). (Indeed, this is how Theorem~\ref{LSthm} is proved, by recursively adding to $M$ witnesses for existential first order statements, and doing this in a minimalistic fashion so as not to make $M$ uncountable.) For example, if $t$ is a sequence of members of $M$ and $x \in M$, then the statement ``$x \notin E_t$" is expressible in first order logic, and this statement is witnessed by some $a \in \mathbb R$ as described above. If $H$ contains a witness to this statement, and $M$ is an elementary submodel of $H$, then $M$ must contain some witness to this statement as well.
In fact, this is precisely how Theorem~\ref{LSthm} enables us to rescue the proof idea outlined earlier, by giving us a countable $M$ that is rich enough to contain witnesses to all statements of the form $x \notin E_t$, where $t$ is a finite sequence of members of $M$. 

If it were possible, we would like to use for $H$ the universe of all sets: this way a statement is true if and only if $(H,\in)$ models that statement, and $H$ contains witnesses to every existential statement.
However, Theorem~\ref{LSthm} states that $H$ should be a set, and the universe of all sets is not a set but a proper class. Indeed, G\"{o}del's Second Incompleteness Theorem tells us that (unless $\mathsf{ZFC}$ is inconsistent), it is impossible to prove from $\mathsf{ZFC}$ alone that there is a set $M$ such that $(M,\in) \Vdash \mathsf{ZFC}$. And if $H$ were the universe of all sets, then this is precisely what $M$ would do.

Fortunately, there is a standard workaround for this problem. We will take the set $H$ to be the set of all sets hereditarily of size $\leq\! \kappa$ for some sufficiently large cardinal $\kappa$. The
structure $(H,\in)$ satisfies all the axioms of $\mathsf{ZFC}$ except for the power
set axiom, and even this fails only for sets $X$ with $|\mathcal P(X)| > \kappa$. This
makes $H$ a good substitute for the universe of all sets. Indeed, if $\kappa$ is
larger than any set we plan to mention in our proof, then $H$ satisfies $\mathsf{ZFC}$ for
all practical purposes. 

\begin{theorem}\label{willProof}
Bob has a winning strategy for \(BG(W)\) if and only if
\(W\) is countable.
\end{theorem}

\begin{proof}
We proceed with the terminology and assumptions established in Remark \ref{willPartial}.

Let $H$ denote the set of all sets that are hereditarily of size $\leq |\mathbb R|$. The important facts about this set $H$ are that it contains all real numbers, all finite sequences of real numbers, all sets of real numbers, and all possible strategies for any instance of Baker's Game. Furthermore, if $E \subseteq \mathbb R$ is well founded if and only if $(H,\in) \Vdash$ ``$E$ is well founded," and similarly for other basic statements about reals and sets of reals. In other words, $(H,\in)$ correctly interprets all the statements about reals and sets of reals that are made in the remainder of this proof.

Let \(M\) be a countable elementary submodel of $H$ containing
\(\sigma\) and containing every rational number. (In fact, let us note that this latter requirement is redundant, by an argument from elementarity. For example, $M$ contains a given rational number $q$ because $q$ is definable without parameters, i.e. there is a statement $\varphi$ of first order logic that is satisfied by the number $q$ and no other, and $M$ contains a witness to the statement $\varphi$, which must be $q$.) 
We claim that Alice has a
successful counter-strategy to $\sigma$, one that plays only real numbers
from the countable set \(M\).

To see this, let \(E=\bigcup\{E_t:t\in(M\cap\mathbb R)^{<\omega}\}\), where \((M\cap\mathbb R)^{<\omega}\) denotes the set of
all the finite sequences of real numbers in \(M\).
This is a countable union of countable sets, hence countable.

Fix \(x\not\in E\). Our goal is to inductively describe a counter-strategy for Alice according to which $x$ remains legal throughout the entire play of the game (assuming Bob plays according to $\sigma$). Suppose that by round $n$ of the game, Alice has previously chosen
\(t=\tuple{a_0,\dots,a_{n-1}}\in(M\cap\mathbb R)^{<\omega}\)
such that \(x\) remains legal; that is, \(a_{n-1}<x<\sigma(t)\)
(where we may say \(a_{-1}=-\infty\) and \(\sigma(\tuple{})=+\infty\)
for the initial case $n=0$).
Since \(x\not\in E_t\), as in Remark \ref{willPartial},
there exists \(a'\in(a_{n-1},x)\) with
\(x<\sigma(t\concat\tuple{a'})\). While this choice would
keep \(x\) legal, as we observed earlier \(a'\) could be any
number in the interval \((a_{n-1},x)\), most of which
don't belong to \(M\). We assumed Alice only plays members of $M$ (in order to ensure $E$ is countable), so we cannot complete our induction by having Alice play $a'$ in round $n$.

To deal with this, we may consider the following inequality.
\[a_{n-1}<a'<x<\sigma(t\concat\tuple{a'})<\sigma(t)\]
Note that \(x\) and \(a'\) (and therefore
\(\sigma(t\concat\tuple{a'})\)) may or may not
belong to \(M\). However, we may choose rational numbers \(p,q\)
to pad between them as follows.
\[a_{n-1}<a'<p<x<q<\sigma(t\concat\tuple{a'})<\sigma(t).\]
With this choice of $p$ and $q$, let us define the set
\[A=\{a\in(a_{n-1},p):\sigma(t\concat\tuple{a})\in(p,\sigma(t))\}.\]

We have already proven \(A\not=\emptyset\), because \(a'\in A\).
And this is where we exploit that \(M\) isn't just some
arbitrary countable set, but an elementary submodel for $H$. 
Because
\(A\not=\emptyset\), and because $H$ is sufficiently large and sufficiently rich, $(H,\in) \Vdash A \neq \emptyset$. 
In other words, $H$ contains a witness (for example $a'$) to the fact that there is a real number satisfying the statement defining $A$.
This fact must also be witnessed
by \(M\). 
That is, there is some \(a_n\in M\) such that
\(a_n\in A\). 
Then this is exactly what we need: \(a_n\) is
a valid choice for Alice from \(M\) that keeps \(x\) legal:
\[a_{n-1}<a_n<x<\sigma(t\concat\tuple{a_n})<\sigma(t).\]
To complete the inductive description of Alice's strategy, we simply require that she play some such $a_n \in M$ in round $n$ of the game.

So finally, given an uncountable \(W\) and strategy \(\sigma\)
for Bob, Alice may choose
any \(x\in W\setminus E\), and apply the above counter-strategy
to defeat \(\sigma\).
\end{proof}

The reader interested in a more careful and detailed introduction
of this technique is directed to \cite{zbMATH04139731}.

\section{Coding Strategies}

Any proof utilizing elementary submodels may be unpacked
into a proof that avoids them; essentially they are a
convenient meta-mathematical tool to abstract away
closing-up arguments, in this case, providing Alice a desired
countable moveset without constructing it explicitly.
However, this particular result is well-suited for another
style of proof common in the theory of infinite games,
that provides another mechanism to wave away any complicated bookkeeping
of Alice's moves.

\begin{definition}
A strategy that is defined only in terms of the most recent
move by each player (and ignores any moves prior to those)
is known as a \term{coding strategy}.
\end{definition}

We will see shortly why the word ``coding'' is used here.
But first, we will strengthen Proposition
\ref{bakerMainBob} in terms of such strategies.

\begin{proposition}
If \(W\) is countable, then Bob has a winning coding strategy
for \(BG(W)\).
\end{proposition}

\begin{proof}
Let \(W=\{w_n:n<\omega\}\). Then during round \(n\), Bob's
winning coding strategy inspects whether \(w_n\) is legal
(which only requires knowledge of each player's most recent
move). If so, Bob chooses \(w_n\); if not, Bob chooses any
legal number. Since this strategy ensures during round \(n\)
that either \(w_n\) is already illegal, or it makes \(w_n\)
illegal by choosing it, we see that no \(w_n\) is left legal at the end
of the game, and thus Bob wins.
\end{proof}

We now will show the converse holds as well. This is done
by taking the idea of Theorem \ref{willProof}, but as the
strategy we consider here is a coding strategy, we will see that no
elementary submodel is necessary to achieve our result.

\begin{theorem}\label{stevenProof}
Bob has a winning coding strategy for \(BG(W)\) if and only if
\(W\) is countable.
\end{theorem}

\begin{proof}
Let \(\sigma\) be an arbitrary coding strategy for Bob,
where \(\sigma(b,a)\) is Bob's next choice given prior choices
\(b\) by Bob and \(a\) by Alice; for Bob's initial choice we
use the convention \(\sigma(a)=\sigma(\infty,a)\). By 
Observation \ref{bobStrat} we may assume \(\sigma\) only chooses rationals.
We proceed by defining a countable
set \(E\) such that for each \(x\not\in E\),
Alice has a counter-strategy such that \(x\) will always remain
legal at the end of the game.

We begin by defining for each \(\beta\in\mathbb Q\cup\{\infty\}\)
and \(q\in\mathbb Q\),
\[E_{q,\beta}=\left\{x\in(q,b):\forall a
\big(a\in(q,x)\Rightarrow x\geq \sigma(\beta,a)\big)\right\}\]
that is,
\(E_{q,\beta}\) contains all \(x\) above some rational \(q\)
and below Bob's most recent move \(\beta\), such that for any
choice \(a\) by Alice keeping \(x\) legal, Bob's strategy
\(\sigma\) responds by making it illegal.
We claim that this set contains no infinite decreasing sequence bounded
above \(q\). To see
this, let \(S\) be such a sequence and let \(a=\inf S\in(q,\beta)\); then
\(a\in(q,s)\) and \(s<\sigma(\beta,a)\)
for some \(s\in S\), showing \(s\not\in E_{q,\beta}\).
It follows that \(E_{q,\beta}\cap[q+\varepsilon,\infty)\) is well-founded and
therefore countable for each \(\varepsilon>0\),
and thus \(E_{q,\beta}\) is countable.

It follows that
\(E=\bigcup\{E_{q,b}:q\in\mathbb Q,b\in\mathbb Q\cup\{\infty\}\}\) is
countable.
Given \(x\not\in E\), we now may describe
Alice's desired counter-strategy.
Suppose we are given previous moves
\(\alpha\in\mathbb R,\beta\in\mathbb Q\cup\{\infty\}\)
of the game with \(x\) legal, that is, \(\alpha<x<\beta\).
Choose \(q\in\mathbb Q\cap (\alpha,x)\).
Then as \(x\not\in E_{q,\beta}\) and \(x\in(q,b)\), there exists
\(a\in(q,x)\) such that \(w<\sigma(\beta,a)\).
Then \(b=\sigma(\beta,a)\in\mathbb Q\) is Bob's next move,
and observe that \(a<x<b\). Since Alice's counter-strategy
ensures \(x\) is kept legal throughout the game,
whenever \(W\) is uncountable
she may choose any \(x\in W\setminus E\) to defeat Bob's strategy.
\end{proof}

In general, having a winning strategy that considers the full
history of a game is not sufficient to have a winning coding
strategy.
However, often a winning coding
strategy will be defined in terms of a
(full-information) winning strategy; this is possible
in cases where the strategy can also \textit{encode}
the history of the game needed by full-information strategy
into the player's own move each round. This is the case for
Bob in Baker's Game.

\begin{theorem}
Bob has a winning strategy for \(BG(W)\) if and only if
Bob has a winning coding strategy.
\end{theorem}

\begin{proof}
First, let
\(x\sim y\Leftrightarrow x-y\in\mathbb Q\). Note that \(\sim\) is
an equivalence relation where each equivalence class \([x]\) is a
copy of \(\mathbb Q\). Let \(\mathcal D\) be the partition
of \(\mathbb R\) defined by \(\sim\), noting
\(|\mathcal D|=|\mathbb R|\) and each \(D\in\mathcal D\)
is dense in \(\mathbb R\).
Finally, let \(\mathbb R^{<\omega}\) collect all finite-length
sequences of real numbers, where \(|\mathbb R^{<\omega}|=|\mathbb R|\).
Thus, we may fix a bijection
\(f:\mathbb R^{<\omega}\to\mathcal D\).

Let \(\sigma\) be a winning strategy for Bob.
We define the coding strategy \(\tau\) as follows.
To define Bob's first move \(\tau(a)\), consider
the move \(b=\sigma(\tuple{a})\) according to the winning
full-information strategy. As \(f(\tuple{a})\)
is dense in \(\mathbb R\), its intersection with the open set \((a,b)\) must be nonempty, so choose \(\tau(a)\in f(\tuple{a})\cap(a,b)\).
Note that in the following round, Bob will be able to decode
\(\tuple{a}\) from his prior move (which he ``sees'' not as \(\sigma(\tuple{a})\) but only as some
real number \(\beta\)) as it will be the unique \(\tuple{a}\in\mathbb R^{<\omega}\)
with \(\beta\in f(\tuple{a})\). Additionally, this choice
\(\tau(a)\) for Bob is strictly less than the choice 
\(b=\sigma(\tuple{a})\) given by the winning full information strategy.

Then to define subsequent moves \(\tau(\beta,a)\) for \(a<\beta\),
let \(t\in\mathbb R^{<\omega}\) be the unique sequence of real numbers
such that \(\beta\in f(t)\). Again consider the move
\(b=\sigma(t\concat\tuple{a})\) that would be made according to the winning full-information
strategy in response to the sequence \(t\concat\tuple{a}\) of Alice's moves so far. As in the base case (but adding the additional requirement that moves must be less than \(\beta\)), Bob may choose
\(\tau(\beta,a)\in f(t\concat\tuple{a})\cap (a,b)\cap(a,\beta)\).
Note that again the sequence of moves \(t\concat\tuple{a}\) has
been successfully encoded into Bob's move \(b'=\tau(\beta,a)\) and may be extracted
in the next round by choosing the unique \(t\concat\tuple{a}\)
with \(b'\in f(t\concat\tuple{a})\).

We then note that by construction, the coding strategy \(\tau\) produces
legal moves for Bob that are strictly less than those that would have
been chosen by \(\sigma\) for the same choices of Alice.
Since \(\sigma\) is winning, so is \(\tau\) by Lemma \ref{bobSmaller}.
\end{proof}

\begin{corollary}
Bob has a winning strategy for \(BG(W)\) if and only if
Bob has a winning coding strategy if and only if
\(W\) is countable.
\end{corollary}

\section{Conclusion}

In summary, the following is now known for ZFC.

\begin{itemize}
\item
    Bob has a winning strategy if and only if
    \(W\) is countable.
\item
    Alice has a winning strategy if and only if \(W\) contains a perfect set.
\item
    Using the Axiom of Choice, one may construct a 
    \term{Bernstein set} of real numbers
    that (in particular)
    is uncountable but fails to contain any perfect set. Thus
    Baker's Game is \term{indetermined} for the Bernstein set.
\end{itemize}

A sense of d\'ej\`a vu may come upon the reader already familiar
with the study of topological game theory. The \term{Banach-Mazur game}
was introduced for problem 43 in the Scottish Book \cite{zbMATH06482040},
used by mathematicians of the Lw\'ow School of Mathematics who frequented
the Scottish Caf\'e, located in the city known now as Lviv, Ukraine.

When played on the real numbers, this game starts with Bob, and proceeds
essentially like Baker's Game, except that they alternate choosing arbitrary
open subintervals of each other, rather than just moving left or right
endpoints of the open interval of legal numbers. As before, Alice
wins provided some member of a payoff set \(W\) remains legal (within
all chosen subintervals). The following facts for this game are
well-known.

\begin{itemize}
\item
    Bob has a winning strategy if and only if \(W\)'s
    intersection with some open set is \term{meager}.
    (Meager means "small" in a certain sense, including all countable sets,
    but meager sets can be uncountable: for example, a
    Cantor set.)
\item
    Alice has a winning strategy
    if and only if \(W\)'s complement is meager.
\item
    A Bernstein set's intersection with each open set is non-meager,
    but its complement
    is also non-meager.
    So the Banach-Mazur game is also indetermined for such a set.
\end{itemize}

So while these games are certainly not identical (when \(W\)
is the Cantor set, Bob wins the Banach-Mazur game while Alice
wins Baker's Game), there are definitely interesting parallels
for how these games can characterize large and small sets.
The Banach-Mazur game has enjoyed much attention over the past
century, and many questions are still being actively investigated
with regards to what kinds of winning strategies may exist, e.g.
\cite{zbMATH07370902}. These questions often consider the
Banach-Mazur game played upon a general topological space.

In that
spirit, we conclude with the following generalization of Baker's
Game that may be worth further study. Let \(W\) be a subset of
some partial (or perhaps linear) order \(P\). Then Baker's Game
is already well-defined for \(P\), except that it may be possible
at some finite stage that no number remains legal - in that case,
Alice loses immediately as no member of \(W\) will remain
legal at the end of the game anyway.

\begin{question}
Does there exist a partial
or linear order with an uncountable payoff set where Bob
has a winning strategy in this generalized Baker's Game?
\end{question}

\bibliographystyle{plain}
\bibliography{sample}

\end{document}